\documentclass[10pt]{article}
\usepackage{amssymb}
\usepackage{amsmath}
\usepackage{amsthm}
\usepackage{graphicx}
\usepackage{graphics}
\usepackage{amsthm}

\input pictex.tex

\theoremstyle{definition}

\newtheorem{thm}{Theorem}[section]

\newtheorem{prop}[thm]{Proposition}

\newtheorem{lem}[thm]{Lemma}

\newtheorem{rem}[thm]{Remark}

\input epsf

\begin{document}

\date{}
\author{Leonardo Dinamarca Opazo \quad \& \quad Andr\'es Navas}

\title{Exact quadratic growth for the derivatives of iterates of \\
interval diffeomorphisms with only parabolic fixed points}
\maketitle

\noindent{\bf Abstract.} We consider $C^2$ diffeomorphisms of a closed interval with only parabolic fixed points. 
We show that the maximal growth of the derivatives of the iterates of such a diffeomorphism is exactly quadratic 
provided it has a non-quadratical tangency to the identity at a fixed point that is topologically repelling on one 
side. Moreover, in absence of such fixed points, the maximal growth of the derivatives of the iterates is subquadratic. 

\vspace{0.2cm}


\vspace{0.2cm}

\noindent{\bf MCS 2020:} 
37C05 
37C10, 
37C15, 
37E05. 

\vspace{0.5cm}

In the beautiful work \cite{PS04}, Polterovich and Sodin prove the following somewhat surprising 
result: for every $C^2$ diffeomorphim of a closed interval having only parabolic fixed points, 
the growth of the derivatives of the iterates is at most quadratic. More precisely, if we let \,
$\Gamma^n (f) := \sup_x \max \{ Df^n (x), Df^{-n}(x) \},$ \, 
then 
$$\limsup_{n \to \infty} \frac{\Gamma^n (f)}{n^2} < + \infty.$$ 
In this article we prove that the quotient above actually converges. To do this, we obtain 
finer information for positive iterates. We hence let $\Gamma^n_+ (f) := \sup_x Df^n (x)$, and 
we study the asymptotic of this sequence.  

In all what follows, we will implicitly assume that the map $f$ we are dealing with is orientation preserving and nontrivial.
We call a {\em component of $f$} any subinterval fixed by $f$ that contains no fixed point in its interior. 
We let $\mathcal{C}(f)$ be the family of components of $f$. Note that, for each $I \in \mathcal{C}(f)$, 
one (and exactly one) of the endpoints is topologically repelling for $f$. We denote it by $r_I$. 

Recall that, by the work of Szekeres \cite{Sz58} 
and Kopell \cite{Ko68}, for every $C^2$ diffeomorphism of a half-closed interval with 
no fixed point inside, there exists a unique $C^1$ vector field whose time-1 map coincides with the diffeomorphism.  
This vector field continuously extends to the closure of the interval, yet the extension is not necessarily smooth 
(see \cite[Chapter IV]{Yo95}). 
If $I$ is a component of $f$, we denote by $X_{I}$ the unique $C^1$ vector field defined on the half-closed interval 
$I^+ := I^{\mathrm{o}} \cup \{r_I\}$ whose time-1 map coincides with the corresponding restriction of $f$.

\vspace{0.3cm}

\noindent{\bf Main Theorem.} 
{\em For every $C^2$ diffeomorphism $f$ of a closed interval having only parabolic fixed points, 
the expression $\Gamma^n_+ (f) / n^2$ converges to a finite limit as $n$ goes to infinity, and 
\begin{equation}\label{eq:limite}
\lim_{n \to \infty} \frac{\Gamma^n_+ (f)}{n^2} = 
\sup_{I \in \mathcal{C} (f)} \left[ \frac{1}{2} D^2 f (r_I) \cdot  \max_{y \in I} |X_{I} (y)| \right].
\end{equation}
}

\vspace{0.3cm}

Since $\Gamma^n (f) = \max \{ \Gamma^n_+ (f), \Gamma_+^n (f^{-1})\}$, the result above implies 
Polterovich-Sodin's theorem (by an application to both $f$ and $f^{-1}$ simultaneously).  
Also, it implies that, if $f$ has a quadratic tangency to the identity at each of its fixed points, 
then $\Gamma^n (f) / n^2$ converges to zero. Let us point out, however, that this was first established by Watanabe 
in \cite{Wa04} via a clever modification of Polterovich-Sodin's arguments. Nevertheless, our proof in this case 
(which is essentially that of Watanabe up to a slight modification in the end) gives extra information. For instance, 
for $f$ with no fixed point in the interior, we prove the vanishing of the limit of  $\Gamma^n_+ (f) / n^2$ if the tangency to the 
identity is quadratic at the topologically repelling fixed point, but not necessarily at the other one. This means that topologically 
contracting fixed points can only yield subquadratic growth of the derivative along iterates. Actually, if such a point is not flat, 
then it yields uniform decay of the derivatives along iterates. Rather surprisingly, this is not the case of flat fixed points. 
We will illustrate this phenomenon with an explicit example in Section \ref{s:ex}.

The new key ingredient to establish the exact quadratic rate growth in the Main Theorem are the Szekeres 
vector fields. It is worth mentioning that this was somehow suggested in \cite{PS04}, where it is explicitly 
proposed to use normal forms around fixed points in the analysis. Of course, this is unavailable for infinitely flat 
fixed points, but fortunately, subquadratic growth in this framework has been already established by Borichev in class $C^{\infty}$ 
\cite{Bo04} and by Watanabe in class $C^2$ \cite{Wa04}. In the non flat case, however, there are still technical problems. 
In particular, normal forms are available only in high differentiability (see \cite{Ta73} for a classical result in class $C^{\infty}$ 
and \cite{EN24} for recent sharp results in finite regularity). Moreover, the use of normal forms allows establishing the order 
of growth of the derivatives along iterates, but does not detect its precise asymptotic. 

The crucial observation is that dealing with the exact growth of derivatives along iterates of an interval diffeomorphism $f$ 
naturally leads to consider the 1-parameter groups of $C^1$ diffeomorphisms that realize $f$ as its time-1 map. When normal 
forms are available, these flows are easily detected. However, these exist even for flat fixed points. In particular, if $f$ has no 
fixed point in the interior, classical work of Szekeres and Kopell shows that there are (at most) two such flows, each of them 
arising by looking $f$ as a diffeomorphism of a half-closed interval (these flows may happen to be different). In our discussion, 
only one of them becomes relevant, namely the one that is associated to the topologically repelling fixed point. This  flow is 
encoded by a $C^1$ (generating) vector field $X$ which hence satisfies $X \circ f = X \cdot Df$. This relation obviously implies 
$Df^n = X \circ f^n /X.$
Studying the growth of $Df^n$ then reduces to study the decay of $X$ near the topologically repelling fixed point as well as the 
behavior of the orbits (of the inverse map $f^{-1}$) around it. Fortunately, the precise asymptotic in the case we need to settle 
is already available, so after some preparation work all fits nicely, and the exact formula (\ref{eq:limite}) arises. Most of this 
is carried out in full detail in Section~\ref{s:new}. Quite remarkably, none of the arguments require more regularity than $C^2$.


\section{On Polterovich-Sodin's growth lemma}

A great deal of the analysis done in \cite{PS04} is based on the following crucial lemma. Although 
it is less relevant for our approach, we will still use it at a key step; see Remark \ref{r:todavia} on this.

\vspace{0.35cm}

\noindent{\bf Lemma [on quasiconvex sequences]}  
{\em Let $(a_n)_{n \geq 0}$ be a sequence of real numbers with $a_0 = 0$ and such that, for all $n \geq 1$, 
the inequality 
\begin{equation}\label{eq:quasi}
2 a_n - a_{n-1} - a_{n+1} \leq C e^{-a_n}
\end{equation}
holds for a certain (fixed) $C \geq 0$. Assume that $\liminf_{n \to \infty} \frac{a_n}{n} = 0$. 
Then, for all $n \geq 1$, 
$$e^{a_n} \leq \left( n \, \sqrt{\frac{C}{2}} + 1 \right)^2.$$
}

\vspace{0.2cm}

Following Polterovich and Sodin, below we establish a general quantitative lemma that incorporates some of Watanabe's remarks 
\cite{Wa04}. Although our analysis mostly relies on the generating vector fields, this will simplify the exposition at several parts. 

\vspace{0.1cm}

\begin{lem} \label{lem:general-entero}
Let $f$ be a $C^2$ diffeomorphism of a closed interval $L$ that has no fixed point in the interior. 
If both endpoints are parabolic then, for all $n \geq 1$, it holds
$$\Gamma^n_+ (f) \leq \left( n \, \sqrt{\frac{C}{2}} + 1 \right)^2,$$
where $C := C' e^{C'}$ , with $C' := |L| \cdot 
\left\| \frac{D^2f}{Df} \right\|_{\infty} .$
\end{lem}

\begin{proof} Assume that $f$ pushes all points to the right, otherwise just conjugate it by the reflexion of $L$. 
For simplicity, denote $a_n := \log \Gamma_+^n (f)$. Let $x_n$ be a point that realizes $\Gamma^n_+(f)$, that is, 
$Df^n (x_n) = \Gamma_+^n (f)$. Then, as in \cite{PS04}, we have 
$$a_n = \log Df^n (x_n) = \sum_{i=0}^{n-1} \log Df (f^i (x_n)),$$
$$a_{n-1} \geq \log Df^{n-1} (f(x_n)) = \sum_{i=1}^{n-1} \log Df (f^i (x_n)),$$
$$a_{n+1} \geq \log Df^{n+1} (f^{-1}(x_n)) = \sum_{i=-1}^{n-1} \log Df (f^i (x_n)).$$
Therefore,
$$2a_n - a_{n-1}-a_{n+1} \leq \log Df (x_n) - \log Df (f^{-1}(x_n)).$$
Since $f$ is a $C^2$ diffeomorphism, 
there exists a point $\xi_n \in [f^{-1}(x_n),x_n]$ such that 
$$\log Df (x_n) - \log Df (f^{-1}(x_n)) = D \log Df (\xi_n) \cdot [x_n - f^{-1} (x_n)] =  \frac{D^2f}{Df} (\xi_n) \cdot [x_n - f^{-1} (x_n)],$$
hence 
$$2a_n - a_{n-1}-a_{n+1} 
\leq  \frac{D^2f}{Df} (\xi_n) \cdot [x_n - f^{-1} (x_n)] 
= \frac{D^2f}{Df} (\xi_n) \cdot \frac{[x_n - f^{-1} (x_n)] }{[f^n (x_n) - f^{n-1} (x_n)]} \cdot [f^n (x_n) - f^{n-1}(x_n)].$$
Also, there is $x_n' \in [f^{-1}(x_n),x_n]$ such that 
$$Df^n (x_n' ) =\frac {[f^n (x_n) - f^{n-1} (x_n)]}{[x_n - f^{-1} (x_n)] } ,$$  
hence 
$$2a_n - a_{n-1}-a_{n+1} 
\leq  \frac{D^2 f}{D f} (\xi_n) \cdot \frac{ [f^n (x_n) - f^{n-1}(x_n)] }{Df^n (x_n')}.$$
Finally, the classical Denjoy's control-of-distortion argument gives 
$$\frac{Df^{n} (x_n)}{Df^n (x_n')} \leq e^{V(f)},$$
where $V(f) := \mathrm{var} (\log Df)$. (See \cite{Na11} for a full treatment of this argument.) Therefore, 
\begin{equation}\label{eq:main}
2a_n - a_{n-1}-a_{n+1} 
\leq e^{V(f)} \cdot \frac{D^2 f}{Df} (\xi_n) \cdot [f^n (x_n) - f^{n-1}(x_n)] \cdot e^{-a_n}.
\end{equation}
Since $| f^n (x_n) - f^{n-1}(x_n) | \leq | L  |$ and 
\begin{equation}\label{eq:cprima}
V(f) = \int_L \left| \frac{D^2 f}{D f} \right| \leq C',
\end{equation}
this gives the announced quasiconvexity of $(a_n)$. Since the endpoints are parabolic, 
one has 
\begin{equation}\label{eq:sublinear}
\lim_{n \to \infty} \frac{a_n}{n} = 0
\end{equation}
(this is an elementary exercise using the continuity of the derivative; 
see also \cite[(1.2)]{PS04} for a more sophisticated -though clarifying- proof). Polterovich-Sodin's lemma on quasiconvex 
sequences then allows concluding the proof.
\end{proof}


\section{A proof of the Main Theorem: no interior fixed point}

We start by establishing the Main Theorem for diffeomorphisms of a closed interval $L := [a,b]$ with no interior fixed point. 
With no loss of generality, we can suppose that $f$ pushes all interior points to the right, 
otherwise we can conjugate it by the reflexion of $L$.


\subsection{The case of a quadratic tangency at the repelling fixed point}
\label{s:problem}

Assuming that $D^2f(a) = 0$, our goal is to prove that 
\begin{equation}\label{eq:aprobar0}
\lim_{n \to \infty} \frac{\Gamma^n_+(f)}{n^2} = 0.
\end{equation}
To do this, we closely follow Watanabe's arguments from \cite{Wa04}, with an important modification in the end. 

Fix $\varepsilon > 0$. For each positive $\mu < 1/|L|$, let 
$$\delta := \max \left\{ \frac{\mu \, |L|}{1-\mu\,|L|}, \, \mu \cdot \left\| \frac{D^2 f }{D f } \right\|_{\infty} \right\}.$$
Also define 
$$C' := |L| \cdot \left\| \frac{D^2 f}{D f}  \right\|_{\infty}  .$$
Fix $\mu >0$ small enough such that $\delta e^{C'} \leq 2 \varepsilon$. Let $\mu' > 0$ be such that 
\begin{equation}\label{eq:dondedonde}
\left| D^2 f (\xi)\right| \leq \mu \quad \mbox{for all } \xi \in [a,a+\mu'],
\end{equation}
and let $N$ be large enough so that $f^{N-1} (a+\mu') > b - \mu$. We claim that the sequence 
$$b_n:= \log \left( \frac{\Gamma_+^{n+N} (f)}{\Gamma^{N}_+(f)} \right) = a_{n+N} - a_{N}, $$ 
satisfies the quasiconvexity inequality 
\begin{equation}\label{eq:quasi-desplazada}
2 b_n - b_{n-1} - b_{n+1} \leq 2 \varepsilon e^{-a_{N}} e^{-b_n}.
\end{equation}
Before checking this note that, since $b_0=0$, by (\ref{eq:sublinear}) and 
the lemma on quasiconvex sequences this implies that, for all $n \geq 1$,  
$$\Gamma^{n+N}_+ (f) 
=e^{ a_{n+N}} 
= e^{a_N} e^{b_n} 
\leq e^{a_N} (n \sqrt{\varepsilon e^{-a_{N}}} + 1)^2 
= (n \sqrt{\varepsilon} + \sqrt{e^{a_N}})^2.$$
This obviously gives
$$\limsup_{n \to \infty} \frac{\Gamma^n_+(f)}{n^2} \leq \varepsilon,$$
and since this is true for every $\varepsilon > 0$, this implies (\ref{eq:aprobar0}).  \\

Now, to prove (\ref{eq:quasi-desplazada}), we refer to (\ref{eq:main}). There are two cases to consider:
\begin{itemize}
\item If $x_{n+N} \leq a+\mu'$ then $\xi_{n+N} \leq a+\mu'$, which in virtue of (\ref{eq:dondedonde}) gives $|D^2 f (\xi_{n+N})| \leq \mu$. 
Since $Df(a)=1$, this implies that $Df(\xi_{n+N}) \geq 1-\mu\,| L |$. By (\ref{eq:main}) and (\ref{eq:cprima}), this gives  
$$2a_{n+N} - a_{n+N-1} - a_{n+N+1} 
\leq e^{C'} \frac{\mu \,  |L|}{1-\mu\, |L|} e^{-a_{n+N}} 
\leq \delta e^{C'} e^{-a_{n+N}} 
\leq 2\varepsilon e^{-a_{n+N}},$$
which immediately yields (\ref{eq:quasi-desplazada}).
\item If $x_{n+N} \geq a+\mu'$ then, by the choice of $N$, we have $f^{N-1} (x_{n+N}) \geq b - \mu$. 
This implies that $f^{n+N}(x_{n+N}) - f^{n+N-1}(x_{n+N}) < \mu$ which, by (\ref{eq:main}) and (\ref{eq:cprima}), gives  
$$2a_{n+N} - a_{n+N-1} - a_{n+N+1} 
\leq e^{C'} \left\| \frac{D^2f}{Df} \right\|_{\infty} \mu \, e^{-a_{n+N}} 
\leq \delta e^{C'} e^{-a_{n+N}} 
\leq 2\varepsilon e^{-a_{n+N}},$$
which again shows (\ref{eq:quasi-desplazada}).
\end{itemize}


\subsection{The case of a nonquadratic tangency at the repelling fixed point}
\label{s:new}

Recall that we are supposing that our $C^2$ diffeomorphism $f$ of the closed interval $L := [a,b]$ 
pushes all interior points to the right (which amounts to saying that $X$ is strictly positive 
on the interior $L^{\mathrm{o}}$ of the interval $L$). 
Here, we also assume that $f$ is not quadratically  tangent to the identity at $a$, which implies that 
$D^2 f (a) > 0$. Under these assumptions, our goal is to prove that 
\begin{equation}\label{eq:limit-spf}
\lim_{n \to \infty} \frac{\Gamma^n_+(f)}{n^2} 
= 
\frac{1}{2} D^2 f (a) \cdot \max_{y \in L} X_{L} (y).
\end{equation}
To do this, we first state a lemma that allows localizing the points $x_n$ that realize the maximal  
derivative for $f^n$. By reasons that will become clear later, we will first establish only a half of it. 
The (much harder) second half will be proved after some further developments.

\begin{lem} \label{lem:dos-partes}
There exists a compact subinterval $[A,B] \subset (a,b)$ such that if $(x_n)$ is any sequence satisfying 
$Df^n (x_n) = \Gamma_+^n (f)$, then $f^n (x_n)$ belongs to $[A,B]$.
\end{lem}

\noindent{\em Proof of the first half of Lemma \ref{lem:dos-partes} (existence of $A > a$).} 
Since $D^2 f(a) > 0$, the derivative $Df$ is strictly increasing on a certain interval $[a,A]$. 
We claim that, for all $n \geq 1$, one necessarily has $f^n (x_n) \geq A$. Assume otherwise. 
Then $[x_n,f^n(x_n)]$ is contained in $[a,A]$, where $Df$ is increasing. This implies that 
$Df (f^n(x_n)) > Df(x_n)$, hence
$$Df^n (f(x_n)) = Df^n (x_n) \cdot \frac{Df (f^n(x_n))}{Df(x_n)} > Df^n (x_n).$$
However, this contradicts the fact that $x_n$ realizes the maximum of the derivative of $f^n$. 
$\hfill\square$

\vspace{0.4cm}

Before passing to the proof of the second half of the lemma, we first need to control the orbits of the inverse map $f^{-1}$. 
This is the content of the next folklore result (which seems to go back to Fatou), for which we include a sketch of proof since 
most (all?) of the proofs in the literature use higher regularity (see for instance \cite{Re13} and  \cite[Proposition 3.2]{EN24} 
for finer results in class $C^3$). 

\medskip

\begin{lem} \label{lem:orbitas}
For every point $y \in L^{\mathrm{o}}$, one has $\lim_{n \to \infty} n \, (f^{-n} (y) - a) = \frac{2}{ D^2 f (a)}$. 
Moreover, the convergence is uniform on compact subintervals of $L^{\mathrm{o}}$.
\end{lem}

\begin{proof} Changing $f$ by its inverse and conjugating by a translation, we are reduced to show the following: for a $C^2$ diffeomorphism 
$g$ of an interval $[0,c]$ with no interior fixed point and which is of the form $g (x) = x - s x^2 + o(x^2)$ around the origin (with $s > 0$), 
one has 
$$\lim_{n \to \infty } n \, g^n (y) = \frac{1}{s}$$
for each  $y \in (0,c)$, and the limit is uniform on compact subintervals. To prove this first note that, since all compact subintervals 
eventually enter into any prescribed neighborhood of the origin, we may argue in such a neighborhood. Now consider the conjugate 
map $G := I \circ g \circ I$, where $I$ denotes the inversion $I(z):= 1/z$. We are thus reduced to show that $G^n (z) / sn$ converges 
to 1 as $n$ goes to infinite, and that this convergence is uniform on compact intervals. To do this, we compute:
$$G(z) = \frac{1}{g(\frac{1}{z})} = \frac{1}{\frac{1}{z} - \frac{s}{z^2} + o \big(\frac{1}{z^2} \big)} 
= \frac{z}{1 - \frac{s}{z} + o \big(\frac{1}{z} \big)} = z + s + o(1).$$
Using this, one easily convinces that $G^n (z) = z + ns + o(n)$, which allows establishing the announced limit. 
Moreover, the convergence is easily seen to be uniform on compact subsets. We leave the details to the reader. 
\end{proof}

The preceding lemma ensures that the growth of $\Gamma^n_+(f)$ is bounded from below by a quadratic sequence, 
as we next state.

\vspace{0.1cm}

\begin{lem} One has 
$$\liminf_{n \to \infty} \frac{\Gamma^n_+(f)}{n^2} \geq \frac{[A-f^{-1}(A)] \cdot D^2f (a)}{2}.$$
\end{lem}

\begin{proof} For each $n \geq 1$, 
there exists $A_n \in [f^{-n-1}(A),f^{-n}(A)]$ such that 
$$Df^n (A_n) = \frac{A - f^{-1}(A)}{f^{-n}(A) - f^{-n-1}(A)},$$
hence 
\begin{equation}\label{eq:crec}
\Gamma_+^n (f) \geq \frac{A - f^{-1}(A)}{f^{-n}(A) - f^{-n-1}(A)}.
\end{equation}
By the previous lemma, $f^{-n-1}(A) = a+  \frac{2}{n \, D^2f (a)} + o (\frac{1}{n})$. 
Since close to $a$ we have 
$$f(x) = f(a) + \frac{D^2f (a)}{2} (x-a)^2 + o ( |x-a|^2 ),$$
we thus conclude 
\begin{eqnarray*}
f^{-n} (A) - f^{-n-1} (A) 
&=& f( f^{-n-1} (A) ) -  f^{-n-1} (A) \\
&=& \frac{D^2f (a)}{2} \big| f^{-n-1}(A) -a \big|^2 + o \, \big( \big| f^{-n-1}(A)-a \big|^2 \big) \\
&=& \frac{D^2f(a)}{2} \left( \frac{2}{n \, D^2f (a)} \right)^2 + o \, \Big(\frac{1}{n^2}\Big) \\
&=& \frac{2}{n^2 \, D^2f (a)}  + o \Big(\frac{1}{n^2}\Big) .
\end{eqnarray*}
Introducing this equality into (\ref{eq:crec}) and passing to the limit finally proves the lemma.
\end{proof}

We can now proceed to show the second part of Lemma \ref{lem:dos-partes}. 

\vspace{0.3cm}

\noindent{\em Proof of the second half of Lemma \ref{lem:dos-partes} (existence of $B < b$).} 
By the previous lemma, there exists $\varepsilon > 0$ such that 
\begin{equation}\label{eq:siempre}
\Gamma_+^n (f) \geq 2\varepsilon \, n^2 \quad \mbox{for all } n \geq 1.
\end{equation} 
Fix $B^* > A$ close enough to $b$ so that 
\begin{equation}\label{eq:condition}
\left\| \frac{D^2f}{Df} \right\|_{\infty} \cdot [b-B^*] 
\cdot \exp \left( \left\| \frac{D^2f}{Df} \right\|_{\infty} \cdot |L| \right) \leq 2\varepsilon.
\end{equation}

Recall that we are dealing with a sequence $(x_n)$ such that $Df^n (x_n) = \Gamma_+^n (f)$, 
for which we let $y_n := f^n (x_n)$. The proof will be completed after showing the next two claims. 

\vspace{0.3cm}

\noindent{\underbar{Claim 1:}} There exists $ B' \in ( B^*,b)$ such that the inequalities $x_n \leq A$ and $y_n \geq B'$ 
cannot simultaneously hold. 

\vspace{0.2cm}

Let $X = X_{L}$ be the Szekeres vector field associated to $f$ on the interval $L^+ = [a,b)$.  We have $X \circ f = X \cdot Df$, 
which implies $X \circ f^n = X \cdot Df^n$ for all $n \geq 1$. Now, let $(f^t)$ be the flow of $X$ (so that $f^1 = f$). Note that 
$X \circ f^t = X \cdot Df^t$ holds for all $t$. Following Sergeraert \cite{Se77}, we compute 
$$\frac{f(x) - x}{X(x)} = \frac{1}{X(x)} \int_0^1 \frac{d}{dt} f^t(x) \, dt = \int_0^1 \frac{X (f^t (x))}{X(x)} \, dt = \int_0^1 Df^t (x) \, dt .$$
Since $Df^t (a) = 1$ for all $t$, the right-side expression converges to $1$ as $x$ goes to $a$ (because the flow 
$(f^t)$ depends in a $C^1$ manner on $x$). We hence conclude that, close to $a$,  
$$X(x) \sim f(x)-x \sim\frac{1}{2} D^2f(a) \cdot (x-a)^2.$$ 
We hence let $A' $ be such that $a < A' < \min \{ A, f^{-1} (B^*) \}$ and 
\begin{equation}\label{eq:abajo-de-A'}
X(x) \geq \frac{1}{3} D^2f(a) \cdot (x-a)^2 \quad \mbox{for all } x \in [a,A'].
\end{equation}

Similarly, let $Y$ be the Szekeres vector field associated to $f$ on the interval $(a,b]$.  
Again, we have $Y \circ f^n = Y \cdot Df^n$. Also,  close to $b$, we have $Y (x) \sim f(x)-x.$ 
Thus, letting  
$$C' := \frac{\max\{ 1, \max_{x \in [A',B^*]} X (x)\}}{\min_{x \in [A',B^*]} Y (x)},$$
we may choose $B' \in (B^*,b)$ such that 
\begin{equation}\label{eq:Y}
Y(y) \leq \min \left\{ \frac{\varepsilon}{C'} , \frac{\varepsilon}{3 C' \cdot D^2f (a)} \right\} \quad \mbox{for all } y \in [B',b].
\end{equation}
Recall from Lemma \ref{lem:orbitas} that $n \, (f^{-n}(A)-a)$ converges to $2 / D^2f(a)$. We may hence choose 
$N$ such that
\begin{equation}\label{eq:donde}
f^{-n} (A) - a \geq \frac{1}{n \cdot D^2f (a)} 
\quad \mbox{for all } n \geq N. 
\end{equation}
Finally, we may slightly increase $B'$ if necessary in order to satisfy $B' > f^{N} (A')$ 
(as well as (\ref{eq:Y})).

\vspace{0.15cm}

Now assume for a contradiction that $x_n \leq A$ and $y_n \geq B'$. There are two cases to consider:
\begin{itemize}
\item If $x_n \geq A'$ then, by the definition of $C'$, we have $1/Y(x_n) \leq C'$, and therefore  
$$Df^n (x_n) = \frac{Y(y_n)}{Y(x_n)} \leq \varepsilon,$$
which strongly contradicts (\ref{eq:siempre}). 
\item If $x_n < A'$, we let $i = i_n \geq 1$ be the (unique) integer such that $f^i (x_n) \in [A',f(A'))$. 
Then we compute 
$$Df^n (x_n) = Df^i (x_n) \cdot Df^{n-i} (f^i (x_n)) = \frac{X (f^i(x_n))}{X (x_n)} \cdot \frac{Y (y_n)}{Y (f^i(x_n))}.$$
Using (\ref{eq:abajo-de-A'}) and (\ref{eq:Y}) and later using the definition of $C'$, we obtain 
$$Df^n (x_n) 
\leq \frac{X (f^i(x_n))}{Y (f^i(x_n))} 
\cdot \frac{\varepsilon}{3 C' \cdot D^2f (a)}
\cdot \frac{3}{D^2f (a) \cdot (x_n-a)^2} 
\leq \frac{\varepsilon}{(D^2f(a))^2 \cdot (x_n-a)^2}. 
$$
Finally, since $B' > f^{N} (A)$, we must have $n \geq N$, which by  (\ref{eq:donde}) implies
$$x_n - a \geq f^{-n} (A) - a \geq \frac{1}{n \cdot D^2f (a)}.$$
However, introducing this in the previous inequality gives 
$$Df^n (x_n) \leq \varepsilon \, n^2,$$
which still contradicts (\ref{eq:siempre}). 
\end{itemize}

\vspace{0.3cm}

\noindent{\underbar{Claim 2}:} The inequality $x_n \geq A$ can hold only for finitely many integers $n$. 

\vspace{0.2cm}

Define the sequence $(a_n^*)$ by 
$$a_n^* := \max_{x \in [A,b]} \max \big\{ \log Df^i (x): 0 \leq i \leq n \big\}.$$
There are two cases to consider:

\begin{itemize}
\item Assume first that this sequence is bounded, say by a constant $C$. Fix an integer $N$ strictly larger 
than $\sqrt{\frac{C}{2\varepsilon}}$. By (\ref{eq:siempre}), if $n \geq N$ then one cannot have $x_n \geq A$, 
which gives the desired conclusion in this case.

\item Assume now that the sequence $(a_n^*)$ diverges, and  
let $x_n^*$ be a sequence of points in $[A,b]$ realizing $a_n^*$, 
that is, such that $\log Df^{i_n}(x_n^*) = a_n^*$ for some $i_n \leq n$. 
We first claim that there exists $N$ such that $x_n^* \geq f(A)$ for all $n \geq N$. 
To show this, let $\Delta := \min_{x \in [A,f(A)]} (f(x) - x)$, and fix $N'$ such that $f^{N'} (A) > b - e^{-V(f)} \cdot \Delta$. 
If $n \geq N'$ then, for any $x^* \in [A,f(A)]$, 
$$f(f^n(x^*)) - f^n(x^*) < b - f^{N'}(A) < e^{-V(f)} \cdot \Delta 
\quad \mbox{ and } \quad 
f(x^*) - x^* \geq \Delta.$$ 
Taking  $\xi^* \in [x^*,f(x^*)]$ such that 
$$Df^n (\xi^*) = \frac{f^n(f(x^*)) - f^n(x^*)}{f(x^*)-x^*},$$ 
this implies 
$$Df^n (\xi^*) \leq \frac{e^{-V(f)} \cdot \Delta}{\Delta} = e^{-V(f)}.$$
Moreover, by Denjoy's control of distortion argument, one has $Df^n(x^*) / Df^n (\xi^*) \leq e^{V(f)}$. Therefore,  
for all $n \geq N'$ and all $x^* \in [A,f(A)]$, 
$$Df^n (x^*) < 1.$$
Since $Df^n (x^*) \leq (\| Df \|_{\infty})^{N'}$ obviously holds for all $n \leq N'$ and all $x^* \in [A,f(A)]$, the divergence 
of $(a_n^*)$ shows that $x_n^*$ cannot belong to $[A,f(A)]$ for $n$ larger than a certain integer $N$.

Next we claim that, for all $n \geq N$, the following inequality holds: 
\begin{equation}\label{eq:estrella}
2a_n^* - a_{n-1}^* - a_{n+1}^* \leq 2 \varepsilon e^{-a_n^*}.
\end{equation}
Indeed,  let $i_n \leq n$ be such that 
$$a_n^* = \log Df^{i_n} (x_n^*).$$
Obviously,
$$a_{n-1}^* \geq \log Df^{i_n - 1} (f(x_n^*)).$$
Moreover, since $x_n^* \geq f(A)$, we have $f^{-1}(x_n^*) \geq A$, hence 
$$a_{n+1}^* \geq \log Df^{i_n + 1} (f^{-1}(x_n^*)).$$
Having this at hand, to prove (\ref{eq:estrella}) 
one may proceed as in the proof of Lemma \ref{lem:general-entero} 
using that $f^{i_n-1} (x_n^*) > B^*$ for $n \geq N$
and introducing (\ref{eq:condition}) in the end.

Now define the sequence $b_n^* := a_{n+N}^* - a_N^*$. Note that $b_0 = 0$ 
and that inequality (\ref{eq:estrella}) gives 
$$2b_n^* - b_{n-1}^* - b_{n+1}^* \leq 2 \varepsilon e^{-a_N^*} e^{-b_n^*}.$$
Moreover, (\ref{eq:sublinear}) easily implies that $a_n^* / n $ converges to 0, which is hence also the 
case of $b_n^*/n$. The lemma on quasiconvex sequences then implies that, for all $n \geq 1$, 
$$e^{b_n^*} \leq \left( n \sqrt{\frac{2 \varepsilon e^{-a_N^*}}{2}} + 1 \right)^2$$ 
that is, 
$$e^{a_{n+N}^*} \leq \left( n \sqrt{\varepsilon} + \sqrt{e^{a_N^*}} \right)^2,$$
which obviously gives 
$$\limsup_{n \to \infty} \frac{e^{a_n^*}}{n^2} \leq \varepsilon.$$
In virtue of (\ref{eq:siempre}), this implies the claim in this case.
\end{itemize}

\vspace{0.2cm}

We are finally in position to conclude the proof of (the second half of) Lemma \ref{lem:dos-partes}. 
Indeed, by Claim 1 above, if $y_n \geq B'$ then $x_n$ must lie in $[A,b]$. However, according to 
Claim 2, this may only happen finitely many times, say for the powers $n_1 < n_2 <  \ldots < n_k$. 
For each of these $n_i$, the derivative $Df^{n_i} (x)$ converges to $1$ as $x$ goes to $b$. Therefore, 
there exists $\delta$ such that any $x$ realizing the maximum of $Df^{n_i}$ satisfies $x \leq b-\delta$. 
The number $B := \max \{ B' , f^{n_k} (b-\delta) \}$ then satisfies all of our requirements. 
$\hfill\square$

\vspace{0.2cm}

\begin{rem} \label{r:todavia} The proof of Section \ref{s:problem} as well as that of the second case in Claim 2 above are the only parts of the proof 
of our Main Theorem where we do not know how to proceed without using Polterovich-Sodin's lemma on quasiconvex sequences. The problem is 
that we may be dealing with an infinitely flat fixed point $b$ and, in this situation, there seems to be no control of the generating vector field $Y$ 
leading to a direct proof of subquadratic growth of derivatives along iterates about it, despite the fact that it is topologically contracting. See the 
example in Section \ref{s:ex} that illustrates this phenomenon. Needless to say, this issue remains somewhat mysterious to us. 
\end{rem}

\vspace{0.2cm}

We are finally in position to finish the proof of the Main Theorem in the present case.  To do this, we consider again the 
the Szekeres vector field $X = X_{L}$ associated to $f$ on the interval $L^+ = [a,b)$.  We have 
\begin{equation}\label{eq:vamos}
\Gamma^n_+ (f) 
= Df^n (x_n) 
= \frac{X (y_n)}{X (x_n)}.
\end{equation}
Also, we know that, for $x$ close to $a$, 
$$X(x) \sim \frac{1}{2} D^2f(a) \cdot (x-a)^2.$$ 
Since $y_n$ belongs to $[A,B]$ for all $n$, Lemma \ref{lem:orbitas} implies that  
$$x_n = f^{-n} (y_n) \sim a + \frac{2}{n \, D^2f (a)}.$$ 
Therefore, 
$$X (x_n) \sim \frac{1}{2} D^2 f (a) \left(  \frac{2}{n \, D^2f (a)} \right)^2 = \frac{2}{n^2 \, D^2f(a)}.$$
Introducing this into (\ref{eq:vamos}) and passing to the limit gives 
\begin{equation}\label{eq:por-arriba}
\limsup_{n \to \infty} \frac{\Gamma^n_+ (f)}{n^2} \leq \lim_{n \to \infty} \frac{\sup_{y \in L} X(y)}{n^2 \, X (f^{-n}(y_n))} 
\leq \frac{1}{2} D^2f (a) \cdot \sup_{y \in L} X(y) .
\end{equation}

Now recall that the vector field $X$ continuously extends to the closure of $L$ by letting $X (b) = 0$ (see \cite[Chapter IV]{Yo95}). 
In particular, there exists a point $y^+ \in L^{\mathrm{o}}$ such that $X(y^+) = \max_{y \in L} X(y)$. If we let $x_n^+ := f^{-n} (y^+)$, then
$$\Gamma_+^n (f) \geq Df^n (x_n^+) = \frac{X (f^n (x_n^+))}{X(x^+_n)} = \frac{X(y^+)}{X (x_n^+)} =  \frac{\max_{y \in L} X(y)}{X (x_n^+)}.$$
Again, \, $x_n^+ = f^{-n} (y^+) = a +  \frac{2}{n \, D^2f (a)} + o \big( \frac{1}{n} \big)$ \, and \, $X(x_n^+) \sim \frac{2}{n^2 \, D^2f(a)}$, \, 
which implies that 
\begin{equation}\label{eq:por-abajo}
\limsup_{n \to \infty} \frac{\Gamma^n_+ (f)}{n^2} \geq \frac{1}{2} D^2f (a) \cdot \max_{y \in L} X(y) .
\end{equation}
Finally, putting (\ref{eq:por-arriba}) and (\ref{eq:por-abajo}) together yields (\ref{eq:limit-spf}).


\section{A proof of the Main Theorem: interior fixed points}

Having proved that the Main Theorem is valid for diffeomorphisms with no interior fixed points, we next proceed to prove it in general. 
Again, we will systematically use Polterovich-Sodin's lemma on quasiconvex sequences, though one can avoid this using generating 
vector fields and Yoccoz' continuity theorem \cite[Chapter IV]{Yo95} together with further -though elementary- computations.

We distinguish two cases, according to whether the diffeomorphism is quadratically tangent to the identity at all fixed points or not. 


\subsection{The case of quadratic tangencies to the identity at all fixed points}

Fix $\varepsilon > 0$, denote again 
$$C' := | L | \cdot \left\| \frac{D^2 f}{Df} \right\|_{_{\infty}},$$ 
and let 
$$\delta := \frac{2 \varepsilon}{e^{C'} \left\| \frac{D^2 f}{Df} \right\|_{\infty}}.$$
There are  only finitely many components of $f$ of length $> \delta$. Denote them by $J_1,\ldots,J_k$. By our 
assumption, for the restriction $f|_{J_i}$ of $f$ to each $J_i$, the result from Section \ref{s:problem} gives
$$\lim_{n \to \infty} \frac{\Gamma^n_+ (f|_{J_i})}{n^2} = 0.$$
This implies that there exists $N=N_{\varepsilon}$ such that
\begin{equation}\label{eq:enlosji}
\frac{Df^n (x)}{n^2} \leq \varepsilon \quad \mbox{ for all } n \geq N \mbox{ and all } x \in J_1 \cup \ldots \cup J_k.
\end{equation} 
Now, if $J$ is any other component of $f$, then $|J| \leq \delta$. By the choice of $\delta$, 
Lemma \ref{lem:general-entero} gives
\begin{equation}\label{eq:enlosotros}
Df^n (x) \leq (n \sqrt{\varepsilon} + 1)^2 \quad \mbox{ for all } n \geq 1 \mbox{ and all } x \in J.
\end{equation}
Putting (\ref{eq:enlosji}) and (\ref{eq:enlosotros}) together, we obtain
$$\limsup_{n \to \infty} \frac{\sup_x Df^n (x)}{n^2} \leq \varepsilon.$$
Since this holds for every $\varepsilon > 0$, we conclude that $\Gamma^n_+(f) / n^2$ converges to 0. 


\subsection{The case of a non-quadratic tangency to the identity at a fixed point}

Here we assume that $f$ has at least one component $I_*$ with a topologically repelling fixed point 
at which $f$ is not quadratically tangent to the identity. By the case of no interior fixed points, we have 
$$\lim_{n \to \infty} \frac{\Gamma_+^n (f|_{I_*})}{n^2} 
= 
\frac{1}{2} D^2 f (r_{I_*}) \cdot \max_{y \in I_*} \big| X_{I_*} (y) \big| > 0.$$
Fix $\varepsilon > 0$ such that 
\begin{equation}\label{eq:porabajo}
\varepsilon < \frac{1}{2} D^2 f (r_{I_*}) \cdot \max_{y \in I_*} \big| X_{I_*} (y) \big|,
\end{equation}
and define again 
$$C' := | L | \cdot 
\left\| \frac{D^2 f}{Df} \right\|_{_{\infty}}, \qquad \delta := \frac{2 \varepsilon}{e^{C'} \left\| \frac{D^2 f}{Df} \right\|_{\infty}}.$$
As in the previous section, there are finitely many components of $f$ of length $>  \delta$, say $J_1,\ldots,J_k$. 
On any other component $J$, inequality (\ref{eq:enlosotros}) holds, which means that 
\begin{equation}\label{eq:cota-afuera}
Df^n (x) \leq \left( n \sqrt{\varepsilon} + 1 \right)^2 
\quad \mbox{for all } x \notin J_1 \cup \ldots \cup J_k 
\mbox{ and all }n \geq 1.
\end{equation}
In particular, by the case of diffeomorphisms with no interior fixed points, for each such $J$ we have  
\begin{equation}\label{eq:enpequeños}
\frac{D^2f (r_J)}{2} \cdot \max_{y \in J} \big| X_{f,J} (y) \big| = \lim_{n \to \infty} \frac{\Gamma^n_+ (f|_J)}{n^2} \leq \varepsilon.
\end{equation}
Since 
$$\lim_{n \to \infty} \frac{\Gamma_+^n (f|_{I_*})}{n^2} > \varepsilon,$$
(\ref{eq:cota-afuera}) implies that 
$$\lim_{n \to \infty} \frac{\Gamma_+^n (f)}{n^2} 
= \max_{1\leq i \leq k} \lim_{n \to \infty} \frac{\Gamma_+^n (f|_{J_i})}{n^2}.$$
By the case with no interior fixed points, this gives
$$\lim_{n \to \infty} \frac{\Gamma_+^n (f)}{n^2} 
= \max_{1 \leq i \leq k} \left[ \frac{1}{2} D^2 f (r_{J_i}) \cdot \max_{y \in J_i} \big| X_{J_i} (y) \big|  \right].$$ 
In virtue of (\ref{eq:porabajo}) and (\ref{eq:enpequeños}), this finally give 
 $$\lim_{n \to \infty} \frac{\Gamma^n_+ (f)}{n^2} = 
\sup_{I \in \mathcal{C} (f)} \left[ \frac{1}{2} D^2 f (r_I) \cdot  \max_{y \in I} \big| X_{I} (y) \big| \right].$$


\section{An example}
\label{s:ex}

To better explain the pertinence of the example that follows, let us first state a result that deals with non infinitely flat fixed points, 
with the constraint that the topologically repelling fixed point is (at least) quadratically tangent to the identity. For simplicity, we state it for 
$C^{\infty}$ diffeomorphisms, yet it easily extends to $C^2$ diffeomorphisms with tame behavior near the fixed points, according 
to the hypothesis below. Also, we will only consider diffeomorphisms with no interior fixed point, leaving the treatment of the 
general case to the reader. The proposition below is very close to a result from \cite{Bo04}, yet slightly more general. Note 
that it still holds for $k=1$, where it becomes a particular case of the Main Theorem.

\begin{prop} Let $f$ be a $C^{\infty}$ diffeomorphism of a closed interval $L$ with no fixed point inside. 
Suppose that $f$ is tangent to the identity exactly up to order $k \geq 2$ at the topologically repelling 
fixed point, and that it is not infinitely flat  at the topologically contracting fixed point. Then 
$$\lim_{n \to \infty} \frac{\Gamma^n_+ (f)}{n^{\frac{k+1}{k}}} 
= \sqrt[k]{\frac{k^{k+1} \cdot Df^{k+1}(r_L)}{(k+1)!}} \cdot \max_{y \in L} \big| X_{L} (y) \big| .$$
\end{prop}

\noindent{\em Sketch of Proof.} Up to two key steps, the proof is very similar to that of the Main Theorem. 
To fix ideas, we assume that $f$ sends all the interior points of $L:= [a,b]$ to the right, hence $r_L = a$.
\begin{itemize}
\item The first difference concerns the asymptotic of the orbits (of the inverse map) near the topologically repelling fixed point. 
Namely, in this case, for all $y \in (a,b)$ one has
$$\lim_{n \to \infty} n^{\frac{1}{k}} \, (f^{-n} (y) - a) = \sqrt[k]{\frac{ (k+1) \cdot (k-1)!}{ Df^{k+1} (a)}},$$
and the convergence is uniform on compact subsets of $(a,b)$.  
The proof is similar to that of Lemma \ref{lem:orbitas}. One first reduces to study germs of the form 
$g(x) = x - sx^{k+1} + o(x^{k+1})$, with $s > 0$. Then, one conjugates by $I_k (z) := 1 / z^{\frac{1}{k}}$ 
to obtain that $G:=I \circ g \circ I^{-1}$ is of the form $ G(z) = z+sk+o(1)$, from where one easily 
deduces the announced asymptotic. 

\item The second difference lies in the growth of the derivatives of the iterates near the topologically contracting fixed point. 
Indeed, in a neighborhood of this point, the derivative must be strictly increasing, hence smaller than 1 in its interior. 
This allows showing that Lemma \ref{lem:dos-partes} still holds, with a much simpler proof for the second half of it. 
Namely, for a point $x_n$ that realizes $\Gamma_+^n (f)$, the point $y_n := f^n (x_n)$ cannot lie on $[f(B),b]$ if $Df$ 
is strictly smaller than 1 on $(B,b)$ and $n$ is large enough. Indeed, assume that $y_n > f(B)$. If, on the one hand, 
$Df$ is strictly larger than 1 on $(a,x_n)$, then $Df^n (f^{-1}(x_n)) > Df^n (x_n)$. If, on the other hand, $Df$ is 
not strictly larger than $1$ on $(a,x_n)$, then there would exist $a'>a$ such that $x_n \geq a'$ for all $n$. Since 
obviously we have $x_n \leq B$ for large enough $n$, this would yield that 
$$Df^n (x_n) = \frac{Y(y_n)}{Y(x_n)} \leq \frac{Y (y_n)}{\min_{x \in [a',B]} Y(x)}$$
converges to $0$ as $n$ goes to infinity, which is impossible.
\end{itemize}
Using the two items above, the proof is concluded via the very same final argument. We leave the details to the reader. 
$\hfill\square$

\vspace{0.35cm}

What we next show is that, quite surprisingly, the proposition above does no longer hold without the hypothesis 
made for the topologically contracting fixed point. The example we present is actually the inverse of the map constructed 
by Watanabe in \cite{Wa04}. The interesting fact is that this inverse map still gives strong (yet subquadratic) infinitesimal 
expansion along the iterates in certain regions close to the topologically contracting fixed point, which is somewhat 
counter-intuitive. For the sake of clearness, we state this as a proposition. 

\vspace{0.1cm}

\begin{prop} \label{prop:ejem} 
For each $k \geq 2$, there exists a $C^\infty$ diffeomorphism $f$ of the interval with no fixed point in the interior and such that: 
\begin{itemize}
\item it is infinitely flat at the topologically contracting fixed point;
\item $k$ is its exact order of tangency to the identity at the topologically repelling fixed point;  
\item for every $0<\tau<2$, one has 
     $$  \limsup_{n\rightarrow \infty}\frac{\Gamma^n_+(f)}{n^{\tau}}=\infty. $$
\end{itemize}
\end{prop}

\vspace{0.1cm}

We give a proof for this for completeness. First, as it was discussed above, what we need to build is a germ of a topologically contracting 
$C^{\infty}$ diffeomorphism that is infinitely flat at the fixed point yet there exists a sequence of points $x_n$ converging to this point  
such that, for all $\tau < 2$, 
\begin{equation}\label{eq:tau} 
\limsup_{n \to \infty} \frac{Df^n (x_n)}{n^{\tau}} = \infty.
\end{equation}
Indeed, one can easily extend this germ to a $C^{\infty}$ diffeomorphism of a closed interval having no interior fixed point and with 
a tangency to the identity of exact order $k$ at the other fixed point. By the previous discussion, such a diffeomorphism satisfies 
the properties of the proposition above. 

For simplicity of the computations, our germ will be defined around the origin. 
We consider the vector field $X(x)=Y(x)+Z(x)$, where for small $x>0$ we let 
$$        Y(x) := \frac{1}{2}\left(\cos\frac{1}{x}-1\right)\exp\left(-\frac{1}{x}\right), \qquad
        Z(x) :=-\exp\left(-\exp\left(\frac{3}{x}\right)\right). 
$$
We extend them to the origin by letting $X(0)=Y(0)=Z(0)=0$. We leave to the reader to check that $X,Y,Z$ are smooth and infinitely 
flat at the origin, and that $X (x) < 0$ for small $x>0$. 
(Note that these are the negative of the vector fields considered in \cite{Wa04}.) We denote by $f^t$ (resp. $h^t$) 
the flow associated to $X$ (resp. $Y$). We claim that there exist a sequence $(t_k)$ of times that goes to infinity and a sequence 
$(c_k)$ of points converging to the origin such that, for large-enough $k$, one has
\begin{equation}\label{eq:reduced}
    D f^{t_k }(c_k) \geq \frac{1}{2}\cdot \frac{{t_k^2}}{\log(t_k)}.
\end{equation}
Although the times $t_k$ are not necessarily integers, the $C^1$ continuity of the flow of $X$ implies (\ref{eq:tau}) 
for the germ of diffeomorphism $f := f^1$ 
just replacing $t_k$ by its integer part. 
    
We are thus reduced to show (\ref{eq:reduced}). We start by giving a lower bound for the derivative of the flow associated to $Y$. 
As before, we denote $h:= h^1$.

\vspace{0.32cm}

\noindent{\underline{Claim 1:}} There exists a sequence of times $(v_k)$ that goes to infinity  
and a sequence of points $(c_k)$ converging to the origin such that 
 \begin{equation} \label{eq:estim-Y}
     Dh^{v_k} (c_k) \geq \frac{v_k^2}{\log (v_k)}.
 \end{equation}

To show this, for each $k\geq 1$ we let $a_k := \frac{1}{2\pi k}$. One readily checks that 
$$    Y(a_k)=0, \qquad DY(a_k)=0, \qquad D^2Y(a_k)=-\frac{1}{2}(2\pi k)^4\exp(-2\pi k) < 0.$$
We denote $\varepsilon_k := D^2Y(a_k).$ \, A straightforward computation shows that there exists $M\in (0,2\pi)$ 
such that, if we denote $b_k := \frac{1}{2\pi k+ M} < a_k$, then for each $k \geq 1$ and all $x\in [b_k,a_k]$, it holds 
\begin{equation}\label{eq:entre-medio}
2\varepsilon_k\leq D^2Y (x)\leq \frac{1}{2}\varepsilon_k.
\end{equation}
For each $k\geq 1$, we consider the vector fields $\bar{Y}_k, \tilde{Y}_k$ around $a_k$ given by  
$\bar{Y}_k(x) := \varepsilon_k(x-a_k)^2$ \, and \, $ \tilde{Y}_k(x) := \frac{\varepsilon_k}{4}(x-a_k)^2.$ \, 
By (\ref{eq:entre-medio}), for all  $x\in[b_k,a_k]$,  
\begin{equation}\label{eq:ineq}
    \bar{Y}_k (x)\leq Y (x) \leq \tilde{Y}_k (x).
\end{equation}
Note that these vector fields induce affine flows, that we denote $\bar{h}^t$ and $\tilde{h}^t$ accordingly. 
In concrete terms, 
\begin{equation}\label{eq:flow}
    \bar{h}^t (x) = a_k+\frac{x-a_k}{1-\varepsilon_k t (x-a_k)},\qquad D \bar{h}^t (x)=\frac{1}{(1-\varepsilon_k t (x-a_k))^2}.
\end{equation}
In particular, 
\begin{equation}\label{estim}
    D \bar{h}^{t} (h^{-t} (b_k)) 
    = \frac{1}{D\bar{h}^{-t}(b_k)}
    = (1+\varepsilon_k t (b_k-a_k))^2 
    \geq \varepsilon_k^2 (a_k-b_k)^2 t^2.
\end{equation}
For each $k\geq 1$, we consider the times
\begin{equation} \label{defin i_k}
    u_k := \exp\left(\frac{1}{4\varepsilon_k^2(a_k-b_k)^2} \right)
\end{equation}
and we let $c_k\in [b_k,a_k]$ be defined as $c_k := \bar{h}^{-u_k}(b_k)$. Since $\varepsilon_k^2(a_k-b_k)^2\rightarrow 0$ as $k$ goes to infinity, 
we have that $u_k\rightarrow \infty$ as $k$ goes to infinity. 
For each $k \geq 1$, we let $v_k,w_k$ be the (unique) times for which 
   $$ h^{v_k} (c_k) = \tilde{h}^{w_k}(c_k) = b_k.$$
Observe that (\ref{eq:ineq}) implies that $w_k\leq v_k \leq u_k.$ 

We can now estimate the derivative of the flow $h^t$. Indeed, using the definition of $c_k$ and inequality (\ref{eq:ineq}), 
we obtain 
\begin{equation}\label{eq:unila}
    Dh^{v_k}(c_k) 
    = \frac{Y (h^{v_k}(c_k))}{Y(c_k)} 
    = \frac{Y(b_k)}{Y(c_k)}
    \geq \frac{\bar{Y}_k(b_k)}{\tilde{Y}_k(c_k)} 
    = 4\cdot\frac{\bar{Y}_k (\bar{h}^{u_k}(c_k))}{ \bar{Y}_k(c_k)}
    = 4\cdot D\bar{h}^{u_k}(c_k). 
\end{equation}
The right-side expression equals \,  $4 \! \cdot \!D \bar{h}^{u_k} (\bar{h}^{-u_k} (b_k))$, 
which, by (\ref{estim}) and (\ref{defin i_k}), satisfies 
 \begin{equation}\label{eq:dorila}
    4 \cdot D \bar{h}^{u_k} (\bar{h}^{-u_k} (b_k))
    \geq 4\cdot\varepsilon_k^2 (b_k-a_k)^2\cdot  u_k^2.
    =\frac{u_k^2}{\log u_k}.
 \end{equation}
Finally, since $u_k \geq v_k$ and the function $x\mapsto\frac{x^2}{\log x}$ 
is increasing, from (\ref{eq:unila}) and (\ref{eq:dorila}) we deduce that 
\begin{equation}\label{eq:asin}
 Dh^{v_k}(c_k)  \geq \frac{v_k^2}{\log (v_k)}.
\end{equation}
This shows (\ref{eq:estim-Y}). 

\vspace{0.32cm}

\noindent{\underline{Claim 2:}} For sufficiently large $k$, we have that $Y(c_k) \leq Z(c_k)$.

\vspace{0.2cm}

Observe that (\ref{eq:ineq}) implies that 
\begin{equation}\label{eq:siseusa}
    Y(c_k)\geq \bar{Y}_k(c_k)=\varepsilon_k(a_k-c_k)^2.
\end{equation}
By the definitions of $\varepsilon_k$ and $u_k$, we have that $\varepsilon_k\rightarrow0$ 
and $u_k(a_k-b_k)\rightarrow\infty$ as $k$ goes to infinity. In particular, for large enough $k$, 
   $$ u_k(a_k-b_k)(1+\varepsilon_k)\geq1.$$
By (\ref{eq:flow}), for large enough $k$, this yields
   $$ a_k-c_k=a_k - \bar{h}^{-u_k} (b_k)=-\frac{(b_k-a_k)}{1+\varepsilon_k u_k (b_k-a_k)}\geq \frac{1}{u_k}\geq \frac{1}{\exp(\exp(4\pi k))}.$$
Using  (\ref{eq:siseusa}), the definition of the vector field $Z$ and the fact that $c_k<a_k,$ one finally obtains 
   $$ Y(c_k)
    \leq \frac{\varepsilon_k}{\exp(2\exp(4\pi k))}
    \leq\frac{-1}{\exp(\exp(6\pi k))} 
    =  Z(a_k)
    \leq Z (c_k).$$

\vspace{0.32cm}

\noindent{ \underline{Claim 3:}} If we let $t_k$ be defined so that $f^{t_k} (c_k) = b_k$, then (\ref{eq:reduced}) holds 
for sufficiently large $k$. 

\vspace{0.2cm}

As before, we have that $t_k\leq v_k$ and $t_k\rightarrow\infty$ as $k$ goes to infinity. Using (\ref{eq:asin}) and 
the previous claim, we can estimate the derivative of the flow of $X$ as follows: 
$$         Df^{t_k} (c_k)
         =\frac{X(b_k)}{X(c_k)}\geq \frac{1}{2}\cdot\frac{Y(b_k)}{Y(c_k)}
         =\frac{1}{2}\cdot D h^{v_k} (c_k)
         \geq \frac{v_k^2}{2\cdot\log v_k}\geq \frac{t_k^2}{2\cdot \log t_k}.
$$


\section{Some questions}

It is a consequence of the Main Theorem (actually, of Polterovich-Sodin's result) that the growth of the second 
derivative of a $C^2$ diffeomorphism of a closed interval having only parabolic fixed points is at most polynomial 
of degree 5. Indeed, using the affine derivative $D^2 / D$ and its chain rule 
$$\frac{D^2 (f \circ g)}{D (f \circ g)} = \frac{D^2 f}{Df} \circ g \cdot Dg + \frac{D^2 g}{D g},$$
we obtain 
$$\frac{D^2 f^n}{D f^n} (x) = \sum_{i=0}^{n-1} \frac{D^2f}{Df} (f^i (x)) \cdot Df^i (x). $$
Therefore, if $\| Df^n \|_{\infty} \leq C n^2$ holds for all $n \geq 1$, then 
\begin{equation}
\left\| \frac{D^2 f^n}{D f^n} \right\|_{\infty} 
\leq \left\| \frac{D^2f}{Df} \right\|_{\infty} \cdot \left( 1 + \sum_{i=1}^{n-1}  C i^2 \right) 
\leq  C' n^3 \left\| \frac{D^2f}{Df} \right\|_{\infty},
\label{eq:se-usa}
\end{equation}
hence 
$$\left\| D^2 f^n \right\|_{\infty} \leq C'' n^3 \| Df^n \|_{\infty} \leq C''' n^5.$$

For a $C^3$ diffeomorphisms we can imitate the computations above using the 
Schwarzian derivative $S:= D^3/D - 3/2 \, (D^2/D)^2$ and its chain rule 
$$S(f \circ g) = S(f) \circ g \cdot (Dg)^2 + S(g).$$
Indeed, this gives 
$$S(f^n) = \sum_{i=0}^{n-1} S(f) \circ f^i \cdot (Df^i)^2,$$
hence 
$$\| S (f^n) \|_{\infty} \leq \| S(f) \|_{\infty} \cdot \left( 1 + \sum_{i=1}^{n-1} (Ci^2)^2 \right) \leq \bar{C}' n^5.$$
Using (\ref{eq:se-usa}) and the definition of $S$, this easily implies that 
$$\| D^3 f^n \|_{\infty} \leq \bar{C}'' n^7.$$

We do not know whether the estimates above are sharp (we suspect they are not). 
We state all of this as part of a more general question.

\vspace{0.35cm}

\noindent{\bf Question 1.} Given a $C^{\infty}$ diffeomorphisms of the closed interval having only parabolic fixed points 
and an integer $k \geq 2$, what is the maximal possible growth for $\| D^k f ^n \|_{\infty}$ as $n$ goes to infinite ? 
In which cases it is achieved ?

\vspace{0.35cm}

Note that, by combining Polterovich-Sodin's estimate with Arbogast's  formula\footnote{This is 
mostly known as Fa\`a di Bruno formula; see \cite{Cr05} on this.}, one easily concludes by induction that 
the growth above is at most polynomial. However, it is unclear whether the degree thus detected is optimal.  

It is worth mentioning that looking at norms different from $\| \cdot \|_{\infty}$ is also pertinent. For example, 
a change of variable in the chain rule for the affine derivative immediately yields\footnote{Note that, since $\frac{D^2 f}{D f} = D (\log Df)$, 
we have $\| \frac{D^2f}{Df} \|_{_1} = \mathrm{var} (\log Df)$, hence the triangle inequality (\ref{eq:triangle-ineq}) is nothing but Denjoy's inequality 
\, $\mathrm{var} (\log D (f \circ g)) \leq \mathrm{var} (\log Df) + \mathrm{var} (\log Dg)$.} 
\begin{equation}\label{eq:triangle-ineq}
\left\| \frac{D^2 (f\circ g)}{D (f \circ g)} \right\|_1 \leq \left\| \frac{D^2 f}{D f} \right\|_1 + \left\| \frac{D^2 g}{D g} \right\|_1,
\end{equation}
hence $\|\frac{D^2 f^n}{Df^n}  \|_{_1}$ grows at most linearly. Actually, there is a nice description of those $f$ for which this growth 
is sublinear: they must have only parabolic fixed points and must arise as the time-1 map of a $C^1$ vector field \cite{EN21}.

If we go to the third derivative along these ideas, a good candidate to replace the 
Schwarzian derivative\footnote{The explicit relation comes from the equality 
$Sf (x) = 6 \, \lim_{y \to x} c(f) (x,y)$, which holds for every $C^3$ diffeomorphism $f$.} 
of a diffeomorphism of a closed interval $L$ is the Liouville cocycle, which 
is the function $c(f)$ defined on $L \times L$ as 
$$c(f) (x,y) := \frac{Df(x) \, Df(y)}{(f(x)-f(y))^2} - \frac{1}{(x-y)^2}.$$
Note that this is defined even if $f$ is $C^1$. However, one can ensure that $c(f)$ belongs to $L^1 (I \times I)$ 
only if $f$ is of class $C^{2+\varepsilon}$ for some $\varepsilon > 0$; see \cite{Na06}. This cocycle satisfies the chain rule 
$$c (f \circ g) (x,y) = c(f) (g(x),g(y)) \cdot Dg(x) \, Dg(y) + c(g) (x,y),$$
which after taking $L^1$ norms gives, via a change of variable, 
$$\| c (f \circ g) \|_1 \leq \| c(f) \|_1 + \| c(g) \|_1.$$
As a consequence, $\| c(f^n) \|_1$ grows at most linearly. 
We do not know any characterization of those $f$ for which this grow is sublinear.


\vspace{0.4cm}

\noindent{\bf Acknowledgments.} The authors are indebted to H\'el\`ene Eynard-Bontemps for her 
interest on this article and several valuable hints and remarks, as well as to the anonymous referee for 
her~/~his pertinent suggestions and corrections. This work was supported by the ECOS project 23003 
``Small spaces under action''. The authors also acknowledge support of the Institut Henri Poincar\'e 
(UAR 839 CNRS-Sorbonne Universit\'e) and LabEx CARMIN (ANR-10-LABX-59-01); in particular, 
the first-named author was funded by the CIMPA-CARMIN program.


\begin{footnotesize}

\vspace{0.1cm}

\noindent Leonardo Dinamarca Opazo \& Andr\'es Navas \\ 

\noindent Universidad de Santiago de Chile\\ 

\noindent Alameda 3363, Santiago, Chile\\ 

\noindent emails: leonardo.dinamarca@usach.cl, andres.navas@usach.cl

\end{footnotesize}

\end{document}